\documentclass[11pt]{article}
\usepackage[utf8]{inputenc}
\usepackage[T1]{fontenc}
\usepackage[english]{babel}
\usepackage{amssymb}
\usepackage{amsmath}
\usepackage{amsthm}
\usepackage{amsfonts}
\usepackage{esint}
\usepackage{dsfont}
\usepackage{wasysym}
\usepackage{bbm}
\usepackage[all]{xy}
\usepackage{hyperref}
\usepackage{caption}
\usepackage{enumitem}
\usepackage{graphicx}
\usepackage{tikz-cd}
\usepackage{tkz-tab}
\usetikzlibrary{decorations.markings,decorations.pathreplacing,positioning,cd}
\usepackage{graphicx}

\usepackage{stmaryrd}

\newtheorem{theorem}{Theorem}[section]

\newtheorem{assumption}[theorem]{Assumption}
\newtheorem{lemma}[theorem]{Lemma}

\newtheorem{proposition}[theorem]{Proposition}

\theoremstyle{definition}
\newtheorem{definition}[theorem]{Definition}

\theoremstyle{remark}

\newtheorem*{remark*}{Remark}

\newcommand{\R}{\mathbb{R}}

\newcommand{\N}{\mathbb{N}}

\newcommand{\Proba}[1]{\mathbb{P}\left( #1 \right)}
\newcommand{\norm}[2]{\left\Vert #1 \right\Vert_{#2}}

\newcommand{\scal}[2]{\left<#1,#2\right>}
\newcommand{\chem}{\text{chem}}
\newcommand{\connects}{\longleftrightarrow}

\title{Shadow and percolation III: chemical distance in continuous landscapes with correlations}

\author{David Vernotte}

\begin{document}
\maketitle

\begin{abstract}
    We study some geometric properties of the excursion set of a slope field $\alpha$ associated to a smooth, planar, centered, Gaussian field $f$. That is, we consider the set $\{\alpha\leq \ell\}$ when $\ell\in \R$ is in the supercritical regime. We show that for almost all such $\ell$, in the sense of the Lebesgue measure, then with high probability the chemical distance between two points connected within $\{\alpha\leq \ell\}$ is comparable to the Euclidean distance between those two points. This result is in the spirit of the Antal and Pisztora theorem \cite{AP96} for Bernoulli percolation. However many new difficulties arise such as the fact that $\alpha$ is a continuous field (not differentiable everywhere) with long range correlations and whose law is still not well understood.
\end{abstract}

\tableofcontents

\section{Introduction}
In this paper, we study some percolation properties of the slope field $\alpha$ associated to some smooth planar Gaussian field $f$. More precisely, we consider $W$ a white noise on $\R^2$. The white noise $W$ is a random centered Gaussian field indexed by functions of $L^2(\R^2)$ such that for any $\varphi_1,\varphi_2\in L^2(\R^2)$ then the covariance between $W(\varphi_1)$ and $W(\varphi_2)$ is simply the scalar product of $\varphi_1$ and $\varphi_2$ in the Hilbert space $L^2(\R^2)$ (we refer to \cite{Jan97} for an introduction of such objects). Let $q:\R^2\to \R$ be a function satisfying the following assumptions.
\begin{assumption}
\label{a:a1}
There exists $\beta>\frac{5}{2}$ such that the following holds.
\begin{enumerate}
    \item The function $q$ is in $\mathcal{C}^4(\R^2)$ and for any $\alpha=(\alpha_1,\alpha_2)\in \mathbb{N}^2$ with $\alpha_1+\alpha_2\leq 4$ then $\partial^\alpha q$ is in $L^2(\R^2).$
    \item We have for any $z\in \R^2$, $q(z)=q(-z).$
    \item There exists a constant $C>0$ such that for any $\alpha=(\alpha_1,\alpha_2)\in \mathbb{N}^2$ with $\alpha_1+\alpha_2\leq 2$ and for any $\norm{z}{}\geq 1$ we have
    $$\norm{\partial^\alpha q(z)}{}\leq \frac{C}{\norm{z}{}^\beta}.$$
    \item We have $(q\ast q)(0)>0$ where $\ast$ denotes convolution (with $q(z)=q(-z)$ this is equivalent to asking that $q$ is not identically equal to $0$.)
\end{enumerate}
\end{assumption}
In the following we let $q$ be a function satisfying Assumption \ref{a:a1}. We denote by $f : \R^2 \to \R$ the following random function.
\begin{equation}
    \forall z\in \R^2,\ f(z) := (q\ast W)(z)=W(q(\cdot-z)).
\end{equation}
It appears that $f$ is a Gaussian field on $\R^2$ which is stationary, centered and whose covariance matrix is given by
\begin{equation}
    \forall z_1,z_2\in \R^2,\ K(z_1,z_2)=\mathbb{E}[f(z_1)f(z_2)]=(q\ast q)(z_1-z_2).
\end{equation}
Since $q$ is $\mathcal{C}^4$ we may assume that, almost surely, $f$ is of class $\mathcal{C}^2$ (up to a modification of the ambient probability space). When $g : \R^n \to \R^m$ is a $\mathcal{C}^1$ function, we define $\tau_g : \R_+\times \R^n\to \R^m$ by
\begin{equation}
    \label{eq:def_tau}
    \tau_g(t,z) = \begin{cases}\frac{g(z+te_1)-g(z)}{t} & \text{ if $t>0$} \\
\frac{\partial g}{\partial z_1}(z) & \text{ if $t=0$}\end{cases}.
\end{equation}

We now introduce $\alpha$ the slope field associated to $f$. It is defined for $z\in \R^2$ by
\begin{equation}
    \label{eq:def_alpha}
  \alpha(z) := \sup_{t\in \R_+}\tau_f(t,z)=\sup_{t>0}\frac{f(z+te_1)-f(z)}{t},  
\end{equation}
where $e_1=(1,0)$ denotes the first vector of canonical basis of the Euclidean space $\R^2$. This field was already introduced and studied in \cite{Shadow1}, \cite{Shadow2}. If we interpret the function $f$ as a landscape of mountains and we imagine a sun at infinity in the direction $e_1$ emitting parallel rays of light making a slope $\ell\in \R$ with the horizontal plane, then $\alpha(z)$ is the minimal slope needed so that an observer located at $z$ (and of height $f(z)$) may see the sun (the light been blocked by other mountains between $z$ and the sun). An object of interest is therefore the \textit{excursion set} $\{\alpha\leq \ell\}$ which represents the set of all points that are lit by the sun when its rays make a slope $\ell$ with the horizontal plane.
\begin{definition}
Let $\ell\in \mathbb{R}$ be a real parameter called the \textit{level}. The \textit{excursion set at level $\ell$} is the random subset of $\R^2$ given by
\begin{equation}
    \{\alpha\leq \ell\}:= \{z\in \R^2\ |\ \alpha(z)\leq \ell\}.
\end{equation}
We say that $\alpha$ \textit{strongly percolates} at level $\ell$ if the probabilities of the existence of a connected component in $\{\alpha\leq \ell\}$ crossing vertical and horizontal rectangles become close to $1$ when the rectangles are large enough. More precisely, we say that $\alpha$ strongly percolates at level $\ell$ if for any non-degenerated rectangle $\mathcal{R}=[a,b]\times [c,d]\subset \R^2$ we have
\begin{equation}
    \label{eq:strongly_percolates_1}
    \min\left(\Proba{\text{Cross}^{\{\alpha\leq \ell\}}_h(\lambda \mathcal{R})}, \Proba{\text{Cross}^{\{\alpha\leq \ell\}}_v(\lambda \mathcal{R})}\right) \xrightarrow[\lambda \to \infty]{}1,
\end{equation}
where $\text{Cross}_h^{\{\alpha\leq \ell\}}(\mathcal{R})$ (resp. $\text{Cross}_v^{\{\alpha\leq \ell\}}(\mathcal{R})$) denotes the event that there exists a connected component of $\{\alpha\leq \ell\}\cap \mathcal{R}$ that intersects both vertical (resp. horizontal) sides of $\mathcal{R}$.
\end{definition}
It appears that the set of $\ell\in \R$ such that $\alpha$ strongly percolates at level $\ell$ is an interval of $\R$ which could a priori be degenerated. However, results in \cite{Shadow2} show that this interval is actually of the form $[\ell_c,\infty[$ or $]\ell_c,\infty[$ for some $\ell_c\in ]0,\infty[$.
\begin{theorem}[\cite{Shadow2}]
Under Assumption \ref{a:a1}, there exists $\ell_c\in ]0,\infty[$ such that for any $\ell>\ell_c$ then $\alpha$ strongly percolates at level $\ell$ and when $\ell<\ell_c$ then $\alpha$ does not strongly percolates at level $\ell$.
\end{theorem}
Moreover, a consequence of the methods and arguments used in \cite{Shadow2} is the following
\begin{proposition}
\label{prop:high_prob_cross} Suppose that Assumption \ref{a:a1} is verified.
Let $\ell>\ell_c$ and $\mathcal{R}=[a,b]\times [c,d]$ be a non-degenerated rectangle. There exist positive constants $c,C$ depending on $\mathcal{R}$ and $\ell$ such that
\begin{equation}
    \label{eq:strongly_percolates_2}
\forall \lambda \geq 1,\ 
    \min\left(\Proba{\emph{Cross}^{\{\alpha\leq \ell\}}_h(\lambda \mathcal{R})}, \Proba{\emph{Cross}^{\{\alpha\leq \ell\}}_v(\lambda \mathcal{R})}\right) \geq 1-Ce^{-c\lambda}.
\end{equation}

\end{proposition}
By classical arguments (see for instance \cite{RV20}, \cite{MV20}, \cite{MRVK23}, \cite{Shadow2}), once one have proven that \eqref{eq:strongly_percolates_2} holds, it follows that if $\ell>\ell_c$ then almost surely there exists a unique unbounded connected component in the set $\{\alpha\leq \ell\}.$ The next natural question is to understand the geometry of this unbounded component. A natural way to access geometric information is through the chemical distance which is defined as the length of the shortest path joining two points and staying in the set $\{\alpha\leq \ell\}.$ We make this precise in the following definition.
\begin{definition}
    \label{def:chem}
    Let $E\subset \R^2$ be a subset.
    \begin{itemize}
        \item For $z,z'\in E$ we denote by $d_{\chem}^E(z,z')\in
    \mathbb{R}_+\cup\{+\infty\}$ the chemical distance between $z$ and $z'$ in $E$. It is defined as
    $$d_\chem(z,z'):= 
    \inf\{\text{length}(\gamma)\ |\ \gamma \in 
    \Gamma^E(z,z')\},$$
    where $\Gamma^E(z,z')$ denotes the set of all continuous and rectifiable paths $\gamma : [0,1]\to E$ such that $\gamma(0)=z$ and $\gamma(1)=z'$, and where $\text{length}(\gamma)$ denotes the Euclidean length of the rectifiable path $\gamma.$
    \item When $C\subset E$ is a connected subset of $E$ we denote by $\text{diam}_\chem^E(C)$ the chemical diameter of $C$ which is defined as
    $$\text{diam}_\chem^E(C) := \sup_{z,z'\in C}d_\chem^E(z,z').$$
    \item When $D\subset E$ is a finite union of disjoint connected components $D=C_1\sqcup C_2 \cdots \sqcup C_n$, we denote by $S^E_\chem(C)$ the sum of all chemical diameter of the connected components of $D$. That is,
    $$S^E(D) := \sum_{i=1}^n \text{diam}_\chem^E(C).$$
    \end{itemize}
\end{definition}
In this paper we prove the following result.

\begin{theorem}
\label{thm:principal}
Suppose that Assumption \ref{a:a1} is verified.
There exists a set $\mathcal{L}\subset \R$ of full Lebesgue measure (that is, the complementary $\R\setminus \mathcal{L}$ has zero Lebesgue measure) such that the following holds. 
For any $\ell\in \mathcal{L}$ such that $\ell>\ell_c$, there exists a constant $C>0$ depending on $\ell$ such that for any $\varepsilon>0$ then as $\norm{z}{}$ goes to $\infty$ we have
\begin{equation}
 \Proba{0 \overset{\{\alpha\leq \ell\}}{\connects} z \text{ and } d^{\{\alpha\leq \ell\}}_\chem(0,z)\geq C\norm{z}{}}=O\left(\norm{z}{}^{-1+\varepsilon}\right)
\end{equation}
\end{theorem}

\paragraph{Comments on Theorem \ref{thm:principal}} We observe that Theorem \ref{thm:principal} is similar in spirit to the result of Antal and Pisztora \cite{AP96}. That is, when $\ell$ is supercritical then with high probability the chemical probability behaves like the Euclidean distance up to some multiplicative constant. One main difference is that our result does not permit us to say that this hold for all $\ell>\ell_c$ but only for almost all $\ell>\ell_c$. This is mainly due to a new technicality that arises in the continuous setting. In fact, while in a discrete case there is a deterministic upper bound for the length of a path that lives in a box (depending on the size of the box), there is no such deterministic upper bound in the continuous setting where the set $\{\alpha\leq \ell\}$ could contort a lot at arbitrary small scales generating an unexpected large chemical distance. There are different ways to deal with this difficulty. One way is to apply a quantitative implicit function theorem to show that in small squares of size $\varepsilon$ (which can be quantified) then the set $\{\alpha=\ell\}$ looks like the graph of a function. This approach was the one used in \cite{Vernotte2}. However, this approach appears difficult to implement in our setting. In fact, our field $\alpha$ is not even of class $\mathcal{C}^1$. Although there exist some variants of an implicit function theorem for weaker classes of functions (such as the class of locally Lipschitz functions), it is hard to implement and to make such results quantitative. Another approach, which works in dimension two only, relies on a Kac-Rice formula to control the length of $\{\alpha=\ell\}\cap B$ where $B\subset \R^2$ is a box. This allows to control the chemical distance generated in the box $B$ (that is to control $S^{\{\alpha\leq \ell\}}(B\cap \{\alpha\leq \ell\})$ with the notations of Definition \ref{def:chem}), and replaces the deterministic argument that we had in the discrete case. This strategy was the one used in \cite{Vernotte1} and as we mentioned it relies on obtaining a Kac-Rice formula for the level set of $\alpha$. In this paper we obtain such a formula, however we only manage to prove it for almost all levels $\ell\in \R$ instead of all $\ell\in \R$. This explains why our Theorem \ref{thm:principal} only holds for almost all $\ell>\ell_c$. Although we believe that such result should in fact hold for all $\ell>\ell_c$. Another difference between our theorem and the one obtained in \cite{AP96} for classical Bernoulli percolation is that instead of an exponential decay of the probability we only obtain a polynomial decay. This is again due to our strategy to handle the local behaviour around a point. A natural way to improve this bound would be to show that the length of $\{\alpha=\ell\}\cap B$ admits moment of higher order. Such a statement was proven under some additional regularity assumptions but only for Gaussian fields \cite{Gass}, \cite{Ancona}. It is unclear to us how to generalise this result to our field $\alpha$. However, we believe that Theorem \ref{thm:principal} should hold for all $\ell>\ell_c$ and with a stronger decay (super-polynomial if not stretched-exponential).

\paragraph{Strategy of the proof and organisation of the paper}

In \cite{ARS14}, the authors present a general strategy to obtain a control about the chemical distance between two points $z$ and $z'$ in a discrete setting. Essentially, by using a renormalization argument it is shown that with high probability there exists $\gamma$ an open path (the equivalent of a path included in $\{\alpha\leq \ell\}$) that starts near $z$ and ends near $z'$ and such that the length of the path $\gamma$ is bounded from above by $C\norm{z-z'}{}$ where $C>0$ is a constant independent from $z$ and $z'$. Then it is possible to connect $z$ and $z'$ to this path without paying a large cost in chemical distance. Indeed, these connections sit inside small boxes of finite size, hence we may simply use a deterministic upper bound to control the length of these connections. We follow this general strategy, however, we need an additional argument to control the length of the connections from $z$ and $z'$ to the path $\gamma$. This is done in Section \ref{sec:2} where we prove a Kac-Rice formula and use it to control the chemical distance generated in a box $B\subset \R^2$. In Section \ref{sec:2} we will highlight the main difficulties to obtain a Kac-Rice formula and show how to solve each of them. In Section \ref{sec:3} we briefly recall the renormalization argument developed in \cite{ARS14} (see also \cite{Vernotte1}) and we show how to obtain Theorem \ref{thm:principal}.

\paragraph{Acknowledgements} I am very grateful to my PhD advisor Damien Gayet who gave me a lot of advice and who also read a first version of this paper.

\section{Local control of nodal lines}
\label{sec:2}
 The objective of this section is to provide a control of the size of the set
\begin{equation}
\label{eq:zl}
    \mathcal{Z}_\ell(\alpha,B) := \alpha^{-1}(\{\ell\})\cap B=\{z\in B\ |\ \alpha(z)=\ell\},
\end{equation}
 where $B$ is a fixed box in $\R^2$, $\ell>0$ is a fixed parameter and $\alpha$ was introduced in \eqref{eq:def_alpha}. More precisely, we denote by $\sigma_\ell(\alpha,B)$ the 1-dimensional Hausdorff measure of the set $\mathcal{Z}_\ell(\alpha, B)$.
 \begin{equation}
     \sigma_\ell(\alpha,B) := \mathcal{H}^1(\mathcal{Z}_\ell(\alpha,B)),
 \end{equation}
 where $\mathcal{H}^1$ denotes the $1$-dimensional Hausdorff measure.
 It appears that $\sigma_\ell(\alpha,B)$ is a random variable which a priori takes values in $\mathbb{R}_+\cup\{+\infty\}$. The main objective of this section is to obtain a control on the tail of this random variable. In the end, this will be achieved via a Kac-Rice formula that allows us to estimate the expectation of $\sigma_\ell(\alpha,B)$ (and therefore gives information on the tail). First of all we will recall the standard Kac-Rice formula, we will then highlight the new difficulties that me must deal with in order to generalise this formula to our field $\alpha$. Each part of this section is then devoted to solving each of these difficulties.

When $g : \R^2\to \R$ denotes for instance a smooth planar Gaussian field, the Kac-Rice formula gives a close form formula for $\mathbb{E}[\sigma_\ell(g,B)]$. This formula (see for instance \cite{KAC43}, \cite{AT07}, \cite{AAJ25}) reads
\begin{equation}
    \label{eq:kac_rice_objective}
    \mathbb{E}[\sigma_\ell(g,B)]=\int_B \mathbb{E}[\norm{\nabla g(z)}{}|g(z)=\ell]\phi_{g(z)}(\ell)dz.
\end{equation}
In the above expression, $\phi_{g(z)}$ denotes the density of the random variable $g(z)$ with respect to the Lebesgue measure and $dz$ denotes the Lebesgue measure on $\R^2$. The starting point to obtain such formula is based on the coarea formula. However, we mention that in our context we need to solve a variety of problems to make \eqref{eq:kac_rice_objective} rigorous, true, and usable for $g=\alpha$.
\begin{enumerate}
    \item First, it appears that $\alpha$ is not a $\mathcal{C}^1$ function, but a locally Lipschitz function instead. Hence we need to be very careful when working with $\nabla \alpha$ which is not defined everywhere (but only almost everywhere).
    \item Second, it is not obvious how to define $\mathbb{E}[\norm{\nabla \alpha(x)}{}|\alpha(x)=u]$ and work with this quantity. In fact, even if one could make sense of $\nabla \alpha(x)$, it is a priori not clear if the random vector $(\alpha(x), \nabla \alpha(x))$ has a density with respect to the Lebesgue measure on $\R^3$. Using arguments from the Malliavin calculus we will show that the random variable $\alpha(0)$ has a density with respect to the Lebesgue measure on $\R$. This will be enough for our purpose but it is still an open question to know if $\nabla \alpha(0)$ itself has a density with respect with the Lebesgue measure (and even harder, if $(\alpha(0),\nabla \alpha(0))$ has a joint density with respect to the Lebesgue measure on $\R^3$).
    \item Once we have obtain some $\ell>0$ for which \eqref{eq:kac_rice_objective} holds for $g=\alpha$, we also need to show that the two sides which are equal in \eqref{eq:kac_rice_objective} are actually finite. This is non trivial since the joint law of $\alpha(0)$ and $\nabla \alpha(0)$ could be very complicated.
\end{enumerate}
The main result of this section is the following.
\begin{proposition}
\label{prop:kac_rice_alpha}
There exists a set $\mathcal{L}\subset \R$ of full measure (that is $\mathbb{R}\setminus\mathcal{L}$ has zero Lebesgue measure) such that the following holds.
For any $\ell\in \mathcal{L}$ and for any ball $B\subset \R^2$ of the form $B=z_0+[-R_0,R_0]^2$ for some $R_0>0$ and $z_0\in \R^2$, then we have
$$\mathbb{E}[\sigma_\ell(\alpha,B)]<+\infty.$$
\end{proposition}

\subsection{Preliminaries results}
First, it will appear that the field $\alpha$ is not globally $\mathcal{C}^1$. However we show that if $z\in \R^2$ is fixed, then almost surely $\alpha$ is differentiable at $z$. By stationarity it is enough to prove the result at $z=0$. Hence we aim to prove the following result.
\begin{proposition}
\label{prop:diff_at_given_point}
Almost surely, there exists a unique random $T\in [0,+\infty[$ such that $\alpha(0)=\tau_f(T,0)$ and $\alpha$ is differentiable at $0$ with
\begin{equation}
    \nabla \alpha(0) = \tau_{\nabla f}(T,0),
\end{equation}
where we recall that $\tau_{\nabla f}$ was defined by \eqref{eq:def_tau}.
\end{proposition}
The proof of Proposition \ref{prop:diff_at_given_point} will be a combination of deterministic arguments from differential calculus together with some results from \cite{AT07}.
We begin with a basic lemma.
\begin{lemma}
\label{lemma:max_tracking_technical}
Let $X : \R_+\times [-1,1]^2 \to \R$ be a $\mathcal{C}^2$ function. We denote by $A : ]-1,1[^2\to \R\cup\{\infty\}$ the function which maps $h\in ]-1,1[^2$ to $A(h) := \sup_{t\geq 0}X(t,h).$
We make the following assumptions:
\begin{enumerate}
    \item there exists a unique $T_0\in [0,\infty[$ such that $A(0)=X(T_0,0)>0$;
    \item if $T_0>0$ then $\frac{\partial^2 X}{\partial t^2}(T_0,0)\neq 0$, otherwise, if $T_0=0$ then $\frac{\partial X}{\partial t}(T_0,0)\neq 0$;
    \item $\forall \varepsilon>0,\ \exists C>0,\ \forall u\in [-1,1]^2,\forall t\geq C,\ |X(t,u)|\leq \varepsilon.$
\end{enumerate}
Then there exists an open subset $V\subset [-1,1]^2$ containing $0$ and a $\mathcal{C}^1$ application $T : V\to \mathbb{R}_+$ such that for any point $h\in V$ we have $A(h)=X(T(h),h)$. Moreover, the map $h\mapsto A(h)$ is differentiable at $h=0$ and we have
$$\nabla A(0) = \nabla_h X(T_0,0).$$
\end{lemma}
\begin{proof}
We first do the proof in the case $T_0>0$ and we later explain how to adapt it in the case $T_0=0$. Consider the map
$$\Phi : (t,h)\in ]0,\infty[\times ]-1,1[^2\mapsto \frac{\partial X}{\partial t}(t,h).$$
Then, since $T_0$ maximizes $t\mapsto X(t,0)$, we see that $\Phi(T_0,0)=0.$ We also remark that under our assumptions $$\frac{\partial \Phi}{\partial t}(T_0,0)=\frac{\partial^2 X}{\partial t^2}(T_0,0)\neq 0.$$ Hence by the implicit function theorem, there exists an open set $U\subset [-1,1]^2$ containing $0$ as well as an open interval $W=]a,b[$ (with $0<a<b<\infty$) containing $T_0$, and there exists map $T : U\to ]a,b[$ with $T(0)=T_0$ and such that for any $(t,h)\in ]a,b[\times U$ we have $\Phi(t,h)=0\Leftrightarrow t=T(h).$
Now we claim that there exists $V\subset U$ an neighborhood of $0$ such that for any $h\in V$, the supremum defining $A(h)$ is attained at the unique point $T(h)$. In fact, we have $X(0,0)<X(T_0,0)$ by the uniqueness of $T_0$. By continuity there exists $V\subset U$ a neighborhood of $0\in \R^2$ small enough such that for all $h\in V$ we have $X(0,h)< X(T_0,h).$ Moreover, by taking $0<\varepsilon<\frac{X(T_0,0)}{4}$ then by our last assumption we see that there exists a constant $C$ such that when $t>C$ and $h\in V$ we also have $X(t,h)<X(T_0,h).$ By compactness, this shows that for each $h\in V$ there exist at least one point $0<S(h)<C$ such that $A(h)=X(S(h),h)$. We thus have $\Phi(S(h),h)=0$. If we want to conclude that $S(h)=T(h)$ we additionally need to show that $S(h)\in [a,b]$. This can be done by a similar argument, using the continuity of $X$ we may assume without loss of generality (up to restricting the neighborhood $]a,b[\times V$ of $(T_0,0)$) that for any $h\in V$ then the supremum of $t\mapsto X(t,h)$ is attained in $]a,b[$.

We now conclude the proof in the case $T_0>0$. We now know that in a neighborhood of $0$ we have $A(h)=X(T(h),h)$. This immediately shows that $A$ is differentiable at $0$ since $X$ and $T$ are of class $\mathcal{C}^1$. Moreover we have by the chain rule for $i\in \{1,2\}$,
$$\frac{\partial A}{\partial h_i}(0) = \frac{\partial X}{\partial t}(T_0,0)\frac{\partial T}{\partial h_i}(0)+\frac{\partial X}{\partial h_i}(T_0,0).$$
Since $\frac{\partial X}{\partial t}(T_0,0)=0$ we ultimately get $\nabla_h A(0)=\nabla_h X(T,0).$
This concludes the proof in the case $T_0>0$.

In the case where $T_0=0$, we argue that in a neighborhood $V$ of $0\in \R^2$ one may set $T(h)=0$ for all $h\in V$. Indeed, we proceed as before by observing that far from the origin the field $X$ takes small values that cannot be the maximum values. And near $0$ we use the fact that $\frac{\partial X}{\partial t}(0,0)$ is nonzero and in fact negative to get the desired result.
\end{proof}
We now recall a lemma from \cite{AT07}.
\begin{lemma}[{{\cite[Lemma 11.2.10]{AT07}}}]
\label{lemma:adler}
Let $K\subset \R^n$ be a compact. Let $\Omega\subset \R^n$ be an open subset containing $K$.
 Let $\Phi : \Omega \to \R^{n+1}$ be a random function which is almost surely of class $\mathcal{C}^1$. Assume that for any $t\in K$ then the random variable $\Phi(t)$ admits a probability density with the respect to the Lebesgue measure on $\R^{n+1}$. Assume also that all these probability densities are bounded in a neighborhood of $0\in \R^{n+1}$ (uniformly in $t\in K$). Then almost surely, there exists no point $t\in K$ such that $\Phi(t)=0.$
\end{lemma}
This result basically states that under mild assumptions one may not expect to find $n$ random reals satisfying $n+1$ random conditions at the same time.
We now are ready to prove Proposition \ref{prop:diff_at_given_point}.
\begin{proof}[Proof of Proposition \ref{prop:diff_at_given_point}]
We first show that almost surely there are no two points $0\leq T_1<T_2$ such that $\alpha(0)=\tau_f(T_1,0)=\tau_f(T_2,0).$
We first prove that almost surely there are no two such points $T_1<T_2$ with $T_1>0$. Indeed, a simple computation shows that any $T>0$ satisfying $\alpha(0)=\tau_f(T,0)$ must verify $\alpha(0)=\frac{\partial f}{\partial z_1}(Te_1)$. Therefore, it is enough to check that there do not exist $0<T_1<T_2$ such that $\Phi(T_1,T_2)=0$ where $\Phi$ is the following random map :
$$
\begin{array}{ccccc}
    \Phi & :&  \mathbb{R}_+^2 &\to& \mathbb{R}^3 \\
     & & (T_1,T_2) & \mapsto & \begin{pmatrix} f(T_2e_1)-f(0)-T_2\frac{\partial f}{\partial z_1}(T_2e_1) \\
     f(T_1e_1)-f(0)-T_1\frac{\partial f}{\partial z_1}(T_1e_1) \\
     \frac{\partial f}{\partial z_1}(T_1e_1)-\frac{\partial f}{\partial z_1}(T_2e_1)
     \end{pmatrix}
\end{array}$$
By restricting the application $\Phi$ to the compact
$$K_n := \left\{(T_1,T_2)\in [0,\infty[^2\ |\ \frac{1}{n}\leq T_1\text{ and } T_1+\frac{1}{n}\leq T_2\text{ and } T_2\leq n\right\},$$ we may apply Lemma \ref{lemma:adler} and see that almost surely there does not exist a solution of $\Phi(T_1,T_2)=0$ in the set $K_n$. Letting $n$ go to infinity we deduce that almost surely there is no solution of $\Phi(T_1,T_2)=0$ with $0<T_1<T_2$. A very similar argument shows that almost surely if $\alpha(0)=\tau_f(0,0)$ then there is no $T>0$ such that $\alpha(0)=\tau_f(T,0).$ This can be done by replacing $\Phi$ in the above argument by $$\tilde{\Phi} : T\mapsto \begin{pmatrix}
f(Te_1)-f(0)-T\frac{\partial f}{\partial z_1}(Te_1) \\
\frac{\partial f}{\partial z_1}(Te_1)-\frac{\partial f}{\partial z_1}(0)
\end{pmatrix}.$$
Now we will apply Lemma \ref{lemma:max_tracking_technical} to $X=\tau_f$. We have already seen that this $X$ satisfies the first assumption of Lemma \ref{lemma:max_tracking_technical} almost surely. Also the third assumption of Lemma \ref{lemma:max_tracking_technical} is also verified almost surely. Indeed, this is a consequence of the Borell-TIS inequality (see for instance Lemma 2.6 in \cite{Shadow2}). It only remains to show that almost surely, $X$ verifies the second assumption of Lemma \ref{lemma:max_tracking_technical}. Let $T_0\geq 0$ be the unique point such that $\alpha(0)=X(T_0,0).$ If $T_0>0$ we must therefore have $\frac{\partial X}{\partial t}(T_0,0)=0$. If we also had $\frac{\partial^2 X}{\partial t^2}(T_0,0)=0$ then $T_0>0$ would be solution of $\Psi(T)=0$ where $\Psi$ is defined by
$$\begin{array}{ccccc}
    \Psi & : & \R_+ & \to & \R^2  \\
     & & T & \mapsto & \begin{pmatrix}\frac{\partial X}{\partial t}(T,0) \\ \frac{\partial^2 X}{\partial t^2}(T,0) \end{pmatrix}.\end{array}$$
    By applying Lemma \ref{lemma:adler}, one can see that almost surely there is no $T>0$ such that $\Psi(T)=0$. Therefore in the case where $T_0>0$, the second assumption of Lemma \ref{lemma:max_tracking_technical} is verified. The proof is even easier in the case $T_0=0$ since verifying the second condition of Lemma \ref{lemma:max_tracking_technical} boils down to verifying that $\frac{\partial f^2}{\partial z_1^2}(0)\neq 0$ which is true almost surely. This concludes the proof since $\nabla_h \tau_f=\tau_{\nabla f}.$
\end{proof}

Although the next Proposition is not technically needed for the proof of Proposition \ref{prop:kac_rice_alpha}, it will be of used later and gives a good intuition of what does the set $\mathcal{Z}_\ell(\alpha,B)$ introduced by \eqref{eq:zl} look like. To understand this proposition it is good to have in mind that if the field $\alpha$ were of class $\mathcal{C}^1$ then we would expect the level sets $\{\alpha=\ell\}$ to be a $1$-dimensional smooth variety included in $\R^2$, that is a collections of topological circles and lines which can be parameterized by $\mathcal{C}^1$ functions. The excursion set $\{\alpha\leq \ell\}$ would be a smooth $2$-dimensional sub-manifold of $\R^2$ whose boundary coincides with $\{\alpha=\ell\}.$ However, $\alpha$ is not $\mathcal{C}^1$ (it is only localyy Lipschitz as we will see later in Lemma \ref{lemma:lipschitz_as}). Therefore, one cannot expect such a good behaviour. Nevertheless, the next proposition shows that the loss of regularity is not too big. 
\begin{proposition}
\label{prop:good_level_set}
Let $\ell>0$ be fixed. Almost surely, for any $z_0\in \R^2$ such that $\alpha(z_0)=\ell$, there exists $r_0>0$ such that the intersection of $\{\alpha =\ell\}$ with the open Euclidean ball of radius $r_0$ centered at $z_0$ can be parameterized by an injective continuous and piecewise $\mathcal{C}^1$ function where $z_0$ is the only possible point where differentiability is lost.
\end{proposition}
\begin{proof}
The proof is quite technical and there are several cases to study one by one. We will focus only the main cases as the other ones can be dealt with in the exact same way with minor adjustments. In order to understand the level set $\{\alpha=\ell\}$ around a point $z_0=(x_0,y_0)$, one needs to understand the local behaviour of $\alpha$ near $z_0$. When $z_0\in \R^2$ is fixed, we see, by the same arguments that the ones used in Proposition \ref{prop:diff_at_given_point}, that almost surely $\alpha$ is differentiable at $z_0$. This relied on the fact that, for fixed $z_0$, almost surely, there exists only one $t\geq 0$ such that $\alpha(z_0)=\tau_f(t,z_0).$ However, since we want to deal with all points $z_0\in \R^2$ simultaneously, such an argument is not sufficient for Proposition \ref{prop:good_level_set}. In fact, there will exist random points $z_0$ for which there are two reals $T_1<T_2$ such that $\alpha(z_0)=\tau_f(T_1,z_0)=\tau_f(T_2,z_0).$ A first step to prove the proposition is to show that two reals $T_1$ and $T_2$ is the maximum we can have, there cannot be three reals or more. More precisely, when $\ell>0$ is fixed we will show that, almost surely, there exists no point $z_0=(x_0,y_0)\in \R^2$ together with $0\leq T_1< T_2< T_3$ such that $\ell=\alpha(z_0)=\tau_f(T_1,z_0)=\tau_f(T_2,z_0)=\tau_f(T_3,z_0).$
In fact, let us prove this situation cannot happen with the additional condition $T_1>0$ (the case $T_1=0$ can be dealt with separately). We consider the following map :
$$\begin{array}{ccccc}
g & : & \R^2 \times \R_+^* &\to& \mathbb{R}^2  \\
     & & (x,y, t) & \mapsto &  \left(\frac{\partial f}{\partial x}(x+t,y)-\ell, f(x+t,y)-f(x,y)-\ell t\right).
\end{array}.$$
As we have already seen, if $T>0$ is such that $\alpha((x,y))=\tau_f(t,(x,y))=\ell$ then $g(x,y,T)=0$.
We thus consider the map
$$\begin{array}{ccccc}
    G & : & \R^2 \times (\R_+^*)^3& \to & \mathbb{R}^6 \\
     & & (x,y,t_1,t_2,t_3) & \mapsto & (g(x,y,t_1),g(x,y,t_2),g(x,y,t_3)) 
\end{array}$$
By restricting the application $G$ to the compact $$K_n= [-n,n]^2\times \left\{(t_1,t_2,t_3)\in (\mathbb{R}_+^*)^3| \frac{2}{n}\leq t_1+\frac{1}{n}\leq t_2 \leq t_3-\frac{1}{n}\leq n\right\},$$
and applying Lemma \ref{lemma:adler} to each $K_n$, we see that almost surely there are no point $(x,y)\in \R^2$ with $\alpha(x,y)=\ell$ and with $0<T_1<T_2<T_3$ that realise the supremum in the definition of $\alpha(x,y)$. As we mentioned earlier this is still the case if we allow $T_1=0$ by a slight adaptation to the argument.

Now, let $(x_0,y_0)\in \R^2$ such that $\alpha(x_0,y_0)=\ell$. There are two cases. The first possibility is that there exists only one $t\geq 0$ such that $\alpha(x_0,y_0)=\tau_f(t,(x_0,y_0))$. In this case, as seen in the proof of Proposition \ref{prop:diff_at_given_point}, then $\alpha$ is locally $\mathcal{C}^1$ at $(x_0,y_0)$. Moreover classical arguments show that under the constraint $\alpha(x_0,y_0)=\ell$ we cannot have also $\nabla \alpha(x_0,y_0)=0$. Therefore the implicit function theorem implies that locally around $(x_0,y_0)$ then $\{\alpha=\ell\}$ looks like the graph of a $\mathcal{C}^1$ function. This concludes the proof of the Proposition in this particular case.

The other possibility, which is much more complicated, is that there exists $0\leq t_1<t_2$ such that $\alpha(x_0,y_0)=\tau_f(t_1,(x_0,y_0))=\tau_f(t_2,(x_0,y_0)).$ In the following we will assume $t_1>0$, but this can be generalized also to case $t_1=0.$ First we remark by the argument above that we have found $(x_0,y_0,t_1,t_2)$ such that $g(x_0,y_0,t_1)=0$ and $g(x_0,y_0,t_2)=0$ therefore we are left with no more degree of freedom in the sense that any other condition of the form $h(x_0,y_0,t_1,t_2)=0$ (with a reasonable $h$) can only happen if it is already implied by $g(x_0,y_0,t_1)=0$ and $g(x_0,y_0,t_2)=0$ (otherwise we could apply Lemma \ref{lemma:adler} and we would find a contradiction). This basically means that apart from $g(x_0,y_0,t_1)=0$ and $g(x_0,y_0,t_2)=0$, everything else is in generic position.

In particular, we may compute the derivative of $g$ with respect to $x$ and $t$ and we see that it is almost surely invertible. More precisely we have
$$D_{x,t}g(x_0,y_0,t_1) = \begin{pmatrix}\frac{\partial^2 f}{\partial x^2}(x_0+t_1,y_0) & \frac{\partial^2 f}{\partial x^2}(x_0+t_1,y_0) \\ \ell-\frac{\partial f}{\partial x}(x,y) & 0 \end{pmatrix}.$$
This matrix is almost surely invertible (in the sense that almost surely, for all $(x_0,y_0)$ which admits $0<t_1<t_2$ such that $\tau_f(t_1,(x_0,y_0))=\tau_f(t_2,(x_0,y_0))=\ell=\alpha(x_0,y_0)$ then this matrix is invertible). This is due to the fact that in generic position we must have $\frac{\partial f}{\partial x}(x_0,y_0)\neq \ell$ and $\frac{\partial^2 f}{\partial x^2}f(x_0+t_1,y_0)\neq 0$. Therefore by the implicit function theorem, we may find a small open interval $I_1\subset \R$ containing $y_0$ and a small open square $S_1=J_1\times K_1$ containing $(x_0,t_0)$ together with $\mathcal{C}^1$ smooth functions $X_0^1 : I_1 \to J_1$ and $T^1 : I_1 \to K_1$ such that
$$\forall y,x,t\in I_1\times J_1\times K_1,\ g(x,y,t)=0 \Leftrightarrow (x,t)=(X^1_0(y),T^1(y)).$$

We can do a similar construction around $(x_0,y_0,t_2)$ and we obtain also $\mathcal{C}^1$ maps $X_0^2 : I_2 \to J_2$ and $T_2 : I_2\to K_2$. We now define $I=I_1\cap I_2$ which is again a small open interval containing $y_0$. Given $y\in I$, we may wonder which of $(X_0^1(y),y)$ or $(X_0^2(y),y)$ indeed belong to $\{\alpha=\ell\}$. We claim that on one side of the interval $I$ then one of this two points will indeed belong to $\{\alpha=\ell\}$ whereas on the other side of the interval the other point will belong to $\{\alpha=\ell\}$. This will be enough to conclude to the proof of the conclusion. To see this, define $X^1(y):= X_0^1(y)+T^1(y)$ and $X^2(y):= X_0^2(y)+T^2(y)$. To known which of $X_0^1(y)$ or $X_0^2(y)$ is indeed the correct trajectory of $\{\alpha=\ell\}$ it is enough to look at the sign of the following quantity :
$$\frac{f(X^2(y))-f(X^1(y))}{X^2(y)-X^1(y)}-\ell.$$
Note that this quantity is equal to $0$ when $y=y_0$. Moreover almost surely, this quantity has a non zero derivative in terms of $y$ which shows that it is positive one one side of the interval $I$ and negative on the other side. This concludes the proof of the proposition.

\begin{figure}
    \centering
    \includegraphics[width=12cm]{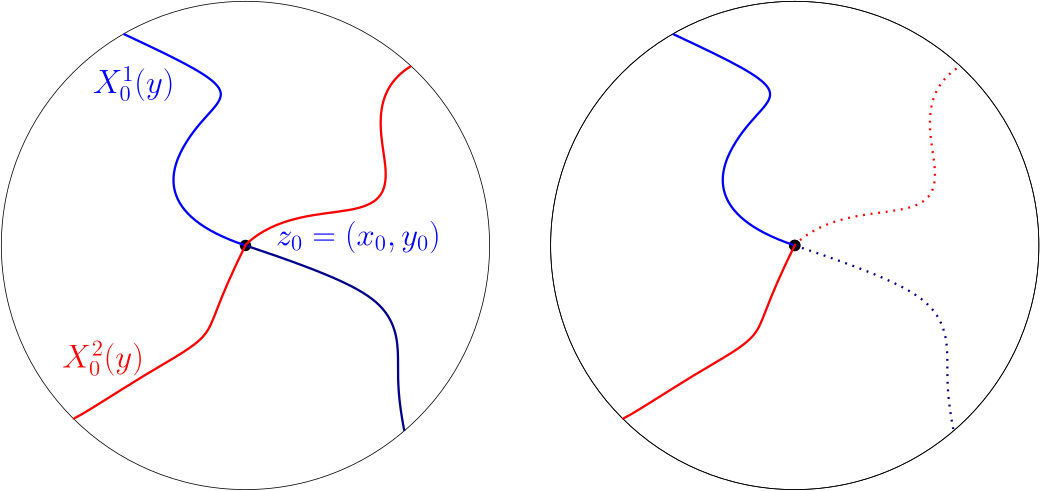}
    \caption{Illustration of the proof of Proposition \ref{prop:good_level_set}. On the left the two $\mathcal{C}^1$ curves $X_0^1$ and $X_0^2$. On the right, the resulting piecewise $\mathcal{C}^1$ curve.}
    \label{fig:piewewise_C1}
\end{figure}

\end{proof}

Our next Lemma shows that almost surely $\alpha$ is locally Lipschitz. This shows in particular that almost surely, $\alpha$ is differentiable almost everywhere in $\R^2$.
\begin{lemma}
\label{lemma:lipschitz_as}
Let $B\subset \R^2$ be a bounded domain. Then almost surely, the function $z\mapsto \alpha(z)$ restricted to $B$ is Lipschitz.
\end{lemma}
\begin{proof}
We only treat the case $B=[0,1]^2$ it is then straightforward to extend this to any bounded domain. By applying the results in \cite{Shadow2} (see in particular Lemma 2.6 in \cite{Shadow2}), then almost surely there exists a finite constant $C>0$ such that for any $z\in B=[0,1]^2$ there exists some $T(z)\in [0,C]$ such that $\alpha(z)=\tau_f(T(z),z)=\frac{f(z+T(z)e_1)-f(z)}{T(z)}$. Moreover if $z,z'\in B=[0,1]^2$ and $T\in [0,C]$ we have
$$|\tau_f(T,z)-\tau_f(T,z')|\leq M\norm{z-z'}{},$$
where $M=:\norm{\nabla^2 f}{[0,C+1]\times [0,1]}$ denotes the supremum norm of the second derivative of $f$ on the compact $[0,C+1]\times [0,1].$ Now, by definition of $\alpha(z')$ we have  $\alpha(z')\geq \tau_f(T(z),z')$ and we find that
$$\alpha(z)-\alpha(z')\leq \tau_f(T(z),z)-\tau_f(T(z),z')\leq M\norm{z-z'}{}.$$
We have a similar upper bound for $\alpha(z')-\alpha(z)$. Since $M$ is almost surely finite, (because $f$ is almost surely a $\mathcal{C}^2$ function and $C$ is finite almost surely) this shows that $\alpha$ is Lipschitz on $B$ and this concludes the proof.
\end{proof}
In the next Lemma we show that some random variables related to $\alpha(0)$ and $\nabla \alpha(0)$ have finite moments of all order. This will be useful to have an effective Kac-Rice formula where both terms that are equal are actually finite. This will also be of use when we will use arguments from the Malliavin calculus to show that the random variable $\alpha(0)$ admits a density with respect to the Lebesgue measure.

\begin{lemma}
\label{lemma:expectation_finite}
For any $u>0$ and any $k\in \mathbb{N}$ we have
\begin{enumerate}
    \item For any $u>0$, $\mathbb{E}\left[\norm{\nabla \alpha (0)}{}^k\mathds{1}_{\alpha(0)>u}\right]<+\infty.$
    \item $\mathbb{E}\left[\alpha(0)^k\right]<+\infty$.
\end{enumerate}
\end{lemma}
\begin{proof}
Let $k\in \mathbb{N}$.
According to Proposition \ref{prop:diff_at_given_point}, we see that we can write
$$\alpha(0)=\frac{f(Te_1)-f(0)}{T}$$
and
$$\nabla \alpha(0) = \frac{\nabla f(Te_1)-\nabla f(0)}{T},$$
for some random $T\geq 0$ (if $T=0$ the two above quantities should be interpreted as a limit when $T$ goes to $0$).
Let $u>0$. We argue that on the event $\{\alpha(0)>u\}$ then is very unlikely to have $T$ very large. Therefore one may control $\nabla{\alpha(0)}$ simply by controlling the norm of the Hessian of $f$ in a bounded region (which will not be to large). To be more precise let $x\geq 1$. By an union bound we have
\begin{align}
    \label{eq:zer1}
    \Proba{\norm{\nabla{\alpha(0)}}{}^k \geq x \text{ and } \alpha(0)>u} \leq &\Proba{\norm{\nabla \alpha(0)}{}^k\geq x \text{ and }T\leq x}\nonumber\\&+\Proba{\alpha(0)>u \text{ and }T\geq x}.
\end{align}
Let us denote $\mathcal{A}_x$ and $\mathcal{B}_x^u$ the two events
\begin{equation*}
    \mathcal{A}_x := \{\norm{\nabla\alpha(0)}{}^k\geq x \text{ and } T \leq x\},
\end{equation*}
\begin{equation*}
    \mathcal{B}_x^u := \{\alpha(0)>u \text{ and } T \geq x\}.
\end{equation*}
On the event $\mathcal{B}_x^u$ there exists $t\geq x$ such that $f(te_1)>f(0)+ut$. By the Borell-TIS inequality, the probability that there exists $t\in [n,n+1]$ such that $f(te_1)>f(0)+un$ is bounded from above by $Ce^{-cu^2n^2}$ where $C,c>0$ are constants. By summing over all $n\geq \lfloor x\rfloor$ we find that there exist constants $C,c>0$ such that
\begin{equation}
\label{eq:zer2}
    \Proba{\mathcal{B}_x^u}\leq Ce^{-cu^2x^2}.
\end{equation}
On the event $\mathcal{A}_x$ then, given the closed form expression of $\nabla \alpha(0)$ in Proposition \ref{prop:diff_at_given_point}, we may find $t\in [0,x]$ such that $\norm{\frac{\nabla f(te_1)-\nabla f(0)}{t}}{}\geq x^{1/k}$. This implies in particular that the supremum norm of the Hessian of $f$ on the compact $[-1,x+1]\times [-1,1]$ must be greater than $x^{1/k}$. By the Borell-TIS inequality an an union bound the probability of this event is upper bounded by $Cxe^{-cx^{2/k}}$ where $C,c$ are two positive constants. We deduce that
\begin{equation}
\label{eq:zer3}
    \Proba{\mathcal{A}_x}\leq Cxe^{-cx^{2/k}.}
\end{equation}
With  \eqref{eq:zer1}, \eqref{eq:zer2} and \eqref{eq:zer3} we see that
\begin{equation*}
    \int_{\R_+}\Proba{\norm{\nabla\alpha(0)}{}^k>x\text{ and }\alpha(0)>u}dx<+\infty.
\end{equation*}
This concludes the proof of the first item of Lemma \ref{lemma:expectation_finite}. The proof of the second item is completely similar. We write
\begin{equation}
    \Proba{\alpha(0)^k \geq x}\leq \Proba{\alpha(0)^k\geq x\text{ and }T\leq x}+\Proba{\alpha(0)^k\geq x \text{ and }T\geq x}
\end{equation}
The second probability may be bounded from above exactly as we did for the event $\mathcal{B}_x$. For the first probability note that on the event $\{\alpha(0)^k\geq x \text{ and } T\leq x\}$ we may find $t\in [0,x]$ such that $\alpha(0)=\frac{\partial f}{\partial z_1}(te_1)$. Therefore on this event we must have $\norm{\nabla f}{\infty}\geq x^{1/k}$ where the supremum norm is again taken on a compact of the form $[-1,x+1]\times [-1,1]$. This event has probability less than $Cxe^{-cx^{2/k}}$ by the Borell-TIS inequality together with an union bound. We conclude the same way that this imply that $\alpha(0)$ has moments of all orders.
\end{proof}

\subsection{Existence of a density for $\alpha(0)$}
We now consider the problem of proving the following result:
\begin{proposition}
\label{prop:density_alpha0}
The random variable $\alpha(0)=\sup_{r>0}\frac{f(re_1)-f(0)}{r}$ admits a density $\phi_0\in L^1(\R)$ with respect to the Lebesgue measure on $\R$.
\end{proposition}
We observe that although it is clear that for any fixed $0<r_1<\dots<r_n$, the random vector $(\frac{f(r_ie_1)-f(0)}{r_i})_{1\leq i \leq n}$ has a density with respect to the Lebesgue measure on $\R^n$, it is a priori not clear how to deduce the existence of a density for $\alpha(0)$ which is a supremum over an infinite set. This problem was already addressed in various contexts and using different techniques (see for instance \cite{AzaisDensity}, \cite{NualartMalliavin}). In order to prove Proposition \ref{prop:density_alpha0} we chose the approach from \cite{NualartMalliavin} (section 2) which uses techniques from Malliavin calculus. We refer the reader unfamiliar with Malliavin calculus to \cite{NualartMalliavin} or \cite{nualart2018introduction}. Recall that we denote by $K(s,t):= \mathbb{E}[f(s)f(t)]$ the covariance function of the field $f$. We define a Hilbert space $H$ as the repoducing Hilbert space associated to this $K$. That is, we let $G$ be the real vector space generated by all functions of the form $(K(s,\cdot))_{s\in \R^2}$, where $K(s,\cdot)$ denotes the function 
$$K(s,\cdot) : t\in \R^2 \mapsto K(s,t)\in \R.$$ We equip $G$ with the following scalar product $$\scal{K(s,\cdot)}{K(t,\cdot)}:= K(s,t).$$
The fact that it is a scalar product comes from the fact that $K$ is a covariance kernel and that the field $f$ is non-degenerated (which is implied by Assumption \ref{a:a1}). We define $H$ as the closure of $G$ for this scalar product. We obtain an Hilbert space $(H,\scal{\cdot}{\cdot}_H).$ In the following we denote by $L^2(\Omega, \R)$ the Hilbert space of square-integrable random variables on $\Omega$ taking values in $\R$ and by $L^2(\Omega, H)$ the Hilbert space consisting of square-integrable random variables taking values in $H$. We denote by $\mathcal{H}_1\subset L^2(\Omega, \R)$ the Hilbert space which is the closure of the vector space generated by all linear combinations of the random variables in the collection $(f(s))_{s\in R^2}.$ The Wiener map is defined as
$$
\begin{array}{ccccc}
    P &:& \mathcal{H}_1& \to& H \\
     & & X & \mapsto & (t\mapsto \mathbb{E}[Xf(t)])
\end{array}.
$$
We remark that $P(f(s))=K(s,\cdot)$ and therefore that $P$ is an isometry. We denote by $D : L^2(\Omega,\R) \to L^2(\Omega,H)$ the Malliavin derivative operator (which is only partially defined on $L^2(\Omega,\R)$), see the first section in \cite{NualartMalliavin}. We mention in particular that for any $s\in \R^2$ we have
\begin{equation}
\label{eq:malliavin_fs}
    D(f(s))=K(s,\cdot).
\end{equation}
Note that in \eqref{eq:malliavin_fs}, the quantity $K(s,\cdot)$ is to be interpreted as an element of $L^2(\Omega, H)$, that is the random variable on $\Omega$ which is deterministic and takes the constant value $K(s,\cdot)\in H.$ Therefore, by linearity, if we set for $t\neq 0$, $X(t)=\frac{f(te_1)-f(0)}{t}$ and for $t=0$, $X(0)=\frac{\partial f}{\partial z_1}(0)$ we find
\begin{equation}
    \label{eq:DXtMalliavin}
    D(X(t))=\frac{K(te_1,\cdot)-K(0,\cdot)}{t}
\end{equation}
\begin{equation}
    \label{eq:DX0Malliavin}
    D(X(0))=\frac{\partial K}{\partial x_1}(0,\cdot).
\end{equation}
We now provide the proof of Proposition \ref{prop:density_alpha0}.

\begin{proof}[Proof of Proposition \ref{prop:density_alpha0}.]
Recall that $X(t)$ was defined for $t> 0$ as
$$X(t)= \frac{f(te_1)-f(0)}{t}.$$
And for $t=0$ by
$$X(0)=\frac{\partial f}{\partial z_1}(0).$$
Note that under Assumption \ref{a:a1}, $X$ is almost surely continuous on $\R_+$.
Let $S=[0,\frac{\pi}{2}]$. For $s\in [0,\frac{\pi}{2}[$ we let $Y(s)$ be defined by
$$Y(s) := X(\tan(s)).$$
We also set $Y(\frac{\pi}{2})=0.$
We remark that since $\alpha(0)>0$ almost surely, we have $\alpha(0)=\sup_{t\geq 0} X(t) = \sup_{s\in [0,\frac{\pi}{2}]}Y(s).$
For convenience we denote by $M := \sup_{s\in [0,\frac{\pi}{2}]}Y(s).$
According to Proposition 2.1.11 in \cite{NualartMalliavin}, to prove that the law of $M$ admits a density with respect to the Lebesgue measure on $\R$, it is enough to verify the following properties:
\begin{enumerate}
    \item Almost surely, $s\mapsto Y(s)$ is continuous on $S=[0,\frac{\pi}{2}].$
    \item $\mathbb{E}[M^2]<+\infty$.
    \item Almost surely, the map $s\mapsto D(Y(s))$ from $S$ to $H$ is continuous.
    \item $\mathbb{E}[\sup_{s\in S}\norm{D(Y(s)}{H}^2]<+\infty$.
    \item Almost surely, for any $s\in S$ such that $Y(s)=M$ then $\norm{D(Y(s))}{H}$ is non-zero.
\end{enumerate}
For the first point we clearly see by composition that almost surely $s\mapsto Y(s)$ is continuous on $[0,\frac{\pi}{2}[$. For the continuity at $\frac{\pi}{2}$ we use the fact that almost surely there exists a random positive constant $C>0$ such that for all $t\geq 0$, $|f(te_1)|\leq C\sqrt{t+1}$ (this is a consequence of the Borell-TIS inequality see also \cite{Shadow2} for a more detailed proof). Therefore, almost surely $X(t)=\frac{f(te_1)-f(0)}{t}\xrightarrow[t\to \infty]{}0$. This proves the continuity of $Y$ at $\frac{\pi}{2}.$
The second item is a direct consequence of Lemma \ref{lemma:expectation_finite} for $k=2$.
For the third item, if $s,t>0$ then by \eqref{eq:DXtMalliavin} we have,
\begin{align*}
    \norm{D(X(t))-D(X(s))}{H}^2 =& \frac{K(t,t)-2K(0,t)+K(0,0)}{t^2} \\
    &+\frac{K(s,s)-2K(0,s)+K(0,0)}{s^2}\\
    &-2\frac{K(s,t)-K(0,s)-K(0,t)+K(0,0)}{st}.
\end{align*}
This quantity clearly converges to $0$ when $s$ converges to $t>0.$
We also have
$$\norm{D(X(t))}{H}^2 = \frac{K(t,t)-2K(0,t)+K(0,0)}{t^2},$$
which clearly converges to $0$ as $t$ goes to infinity. This shows the continuity at $\frac{\pi}{2}$ of $s\mapsto D(Y(s)).$
Finally for $t>0$ we have by \eqref{eq:DXtMalliavin} and \eqref{eq:DX0Malliavin},
\begin{align*}
    \norm{D(X(0)) - D(X(t))}{H}^2 =& \frac{\partial^2 K}{\partial x_1 \partial y_1}(0,0) -2\frac{\frac{\partial K}{\partial x_1}(t,0) -\frac{\partial K}{\partial x_1}(0,0)}{t}\\&+\frac{2K(0,0)-2K(0,t)}{t^2}.
\end{align*}
By assumption we have a $\mathcal{C}^2$ function $\kappa : \R^2 \to \R$ such that $K(x,y)=\kappa(x-y)=\kappa(y-x)$ (one may take $\kappa=q\ast q$). Therefore we see that
$$\frac{\partial K}{\partial x_1}(x,y)=\frac{\partial \kappa}{\partial x_1}(x-y)=-\frac{\partial \kappa}{\partial x_1}(y-x).$$
Therefore $\frac{\partial K}{\partial x_1}(0,0)=0.$
We also see that
$$\frac{\partial^2 K}{\partial x_1^2}(x,y)=\frac{\partial^2\kappa}{\partial x_1^2}(x-y) = -\frac{\partial^2 K}{\partial x_1\partial y_1}(x,y).$$
Using these relations and doing a series expansion of $K(0,t)$ near $0$ we see that as $t$ goes to $0$ we have
$$\norm{D(X(0)) - D(X(t))}{H}^2\xrightarrow[t\to 0]{}0.$$
This concludes the proof of the third item.
For the fourth item we have seen that $t\mapsto \norm{D(X(t))}{H}^2$ is deterministic and continuous in $t$ and that it converges to $0$ when $t$ goes to infinity. Therefore there is nothing more to check for the fourth item.
For the fifth item we note that as soon as $s\in [0,\frac{\pi}{2}[$ we have $\norm{D(X(t))}{H}^2\neq 0$. Therefore it is enough to check that almost surely we have $\alpha(0)>0$, which is indeed the case as mentioned earlier.
All the assumptions are verified and therefore the random variable $M=\alpha(0)$ admits a density with respect to the Lebesgue measure on $\R$.
\end{proof}

\subsection{Kac-Rice formula}
We now glue together the different arguments to prove Proposition \ref{prop:kac_rice_alpha}.

\begin{proof}[Proof of Proposition \ref{prop:kac_rice_alpha}]
We will prove a Kac-Rice formula for the field $\alpha$. The starting point is the coarea formula. Usually it is stated for function which are $\mathcal{C}^1$ however this coarea also holds for function that are Lipschitz (hence differentiable almost everywhere). Let $B=[0,1]^2\subset \R^2$. By applying Theorem 3.2.1 of \cite{Fed14} we see that if $g : B\to \R$ is Lipschitz, then for any test function $h : \R \to \R_+$ which is measurable we have
\begin{equation}
    \int_{B}h(g(z))\norm{\nabla g(z)}{}dz = \int_\R h(u)\sigma_u(g,B)du,
\end{equation}
where the equality is in $\R_+\cup\{+\infty\}.$
We observe that by Lemma \ref{lemma:lipschitz_as} then almost surely the function $\alpha$ is Lipschitz on $B$. Therefore almost surely, for any measurable test function $h : \R \to \R_+$ we have
\begin{equation}
    \int_{B}h(\alpha(z))\norm{\nabla \alpha(z)}{}dz = \int_\R h(u)\sigma_u(\alpha,B)du.
\end{equation}
Thus, we may take expectation on both sides. By interverting the order of integration by the Fubini-Tonelli Theorem (since everything is positive) we obtain
\begin{equation}
\label{eq:21}
    \int_B \mathbb{E}[h(\alpha(z))\norm{\nabla \alpha(z)}{}]dz = \int_\R h(u)\mathbb{E}[\sigma_u(\alpha,B)]du.
\end{equation}
Now since the field $\alpha$ is stationary we see that for any $z\in \R^2$ we have $\mathbb{E}[h(\alpha(z))\norm{\nabla \alpha(z)}{}]=\mathbb{E}[h(\alpha(0))\norm{\nabla \alpha(0)}{}].$
Let us write $\mu(du,dv,dw)$ the probability measure of $(\alpha(0), \frac{\partial \alpha}{\partial z_1}(0), \frac{\partial \alpha}{\partial z_2}(0))$. As mentioned previously, it is not clear if $\mu$ admits a density with respect to the Lebesgue measure on $\R^3$. However due to Proposition \ref{prop:density_alpha0} we see that $\alpha(0)$ itself has a law which admits a density $\varphi_0$ with respect to the Lebesgue measure on $\R$ (we denote by $\mu_{0}(du)=\varphi_0(u)du$ the probability measure associated to $\alpha(0)$. Since $\R^3$ is a Polish space, we may apply the disintegration theorem (see \cite{fremlin2000measure} for instance) and we see that there exists a collection $(\nu_u)_{u\in \R}$ of measures on $\R^2$ which is defined $\mu_0$ almost everywhere (and unique up to a modification in a set of of measure $0$ for $\mu_0$), such that for any measurable test function $\tilde{h} : \R^3\to \R_+$ we have
\begin{equation}
    \int_{\R^3}\tilde{h}(u,v,w)\mu(du,dv,dw) =\int_\R \left(\int_{\R^2}\tilde{h}(u,v,w)\nu_u(dv,dw)\right)\phi_0(u)du.
\end{equation}
We define for any $u\in \R$,
$$\mathbb{E}\left[\tilde{h}\left(\alpha(0),\frac{\partial \alpha}{\partial x_1}(0), \frac{\partial \alpha}{\partial x_2}(0)\right) \ \middle|\  \alpha(0)=u\right] := \int_{\R^2}\tilde{h}(u,v,w)\nu_u(dv,dw).$$ We can then rewrite
\begin{align*}
    \mathbb{E}&\left[\tilde{h}\left(\alpha(0),\frac{\partial \alpha}{\partial x_1}(0), \frac{\partial \alpha}{\partial x_2}(0)\right)\right] \\
    &= \int_\R \mathbb{E}\left[\tilde{h}\left(\alpha(0),\frac{\partial \alpha}{\partial x_1}(0), \frac{\partial \alpha}{\partial x_2}(0)\right)\ \middle|\  \alpha(0)=u\right]\phi_0(u)du.
\end{align*}
Applying this to $\tilde{h}(u,v,w)=h(u)\sqrt{v^2+w^2}$ we find that
\begin{equation}
\label{eq:24}
    \mathbb{E}[h(\alpha(0))\norm{\nabla \alpha(0)}{}] = \int_\R h(u)\mathbb{E}[\norm{\nabla \alpha(0)}{}\ |\  \alpha(0)=u]\phi_0(u)dx.
\end{equation}
Therefore, by \eqref{eq:21} and \eqref{eq:24} together with stationarity, we see that for any test function $h : \R \to \R_+$ we have
\begin{equation}
    \int_\R h(u)\mathbb{E}[\sigma_u(f,B)]du = \text{Vol}(B)\int_\R h(u)\mathbb{E}[\norm{\nabla \alpha(0)}{}|\alpha(0)=u]\phi_0(u)du.
\end{equation}
In the above $\text{Vol}(B)=1$ denotes the Lebesgue measure of the box $B=[0,1]^2$. By varying the test function $h$, we conclude that there exists a measurable set $\mathcal{L}_0\subset \R$ such that the Lebesgue measure of $\R\setminus \mathcal{L}_0$ is zero and such that,
\begin{equation}
    \forall u\in \mathcal{L}_0,\ \mathbb{E}[\sigma_u(f,B)]= \text{Vol}(B)\mathbb{E}[\norm{\nabla \alpha(0)}{}|\alpha(0)=u]\phi_0(u).
\end{equation}
The above equality is a priori an equality in $\mathbb{R_+}\cup\{+\infty\}.$ We now show that for Lebesgue almost all $u\in \R$ we have $\mathbb{E}[\norm{\nabla \alpha(0)}{}|\alpha(0)=u]\phi_0(u)<\infty$. This will conclude the proof.
Observe by Lemma \ref{lemma:expectation_finite} that we have for all $\ell>0$, $$\mathbb{E}[\norm{\nabla \alpha(0)}{}\mathds{1}_{\alpha(0)>\ell}]<+\infty.$$ This rewrites as
\begin{equation}
    \int_\ell^\infty \mathbb{E}[\norm{\nabla \alpha(0)}{}|\alpha(0)=u]\phi_0(u)du<+\infty.
\end{equation}
Therefore we deduce that for almost all $u>\ell$ we have $$\mathbb{E}[\norm{\nabla \alpha(0)}{}|\alpha(0)=u]\phi_0(u)<+\infty.$$ Taking $\ell=2^{-n}$ and letting $n$ go to infinity concludes the proof in the case $B=[0,1]^2$. For a general bounded $B\subset \R^2$ we may cover $B$ by finitely many boxes of the form $z+[0,1]^2$ and noting that $\sigma_u(\alpha,B)$ can be upper bounded by the sum of the $\sigma_u(\alpha,z+[0,1]^2)$ which all have finite expectation by stationarity.
\end{proof}

\section{Control of the chemical distance}
\label{sec:3}
In this section we prove Theorem \ref{thm:principal}. As mentionned in the introduction the proof can be split into two parts. In the first part we show that if $\ell>\ell_c$, then with high probability there exists a path in $\{\alpha\leq \ell\}$ that goes near $0\in \mathbb{R}^2$ and near a given $x\in \mathbb{R}^2$ and which is of length at most $C\norm{x}{}$ where $C$ is a constant that does not depend on $x$ (but which depends on $\ell$). Then in the second part we use our results in Section \ref{sec:2} obtain paths from $0$ and $x$ to this path $\gamma$ while still controlling the Euclidean length of these connections.

\subsection{Construction of a global structure}
We begin by explaining how to obtain a path $\gamma$ going near two points $0$ and $x$ while still controlling this distance. For this we use the strategy developed in \cite{ARS14} which crucially relies on a renormalization argument.
First of all, we introduce some definitions.
\begin{definition}
\label{def:global_structure}
Let $z\in \R^2$, $\varepsilon>0$, $\ell>0$ and $C>0$. We define the event $\mathcal{G}(z,\varepsilon,C,\ell)$ as the event that there exists a continuous and rectifiable path $\gamma : [0,1]\to \R^2$ such that $\gamma$ takes values in $\{\alpha\leq \ell\}$ with $\norm{\gamma(0)-0}{}\leq \norm{z}{}^\varepsilon$, $\norm{\gamma(1)-z}{}\leq \norm{z}{}^\varepsilon$ and such that the Euclidean length of $\gamma$ is at most $C\norm{z}{}.$
\end{definition}
We aim to prove the following result:
\begin{proposition}
\label{prop:global_struct_high_prob}
For any $\ell>\ell_c$, there exists a constant $C>0$ such that for any $\varepsilon>0$ we have a constant $c>0$ such that when $\norm{z}{}\geq 1$ then
\begin{equation}
    \Proba{\mathcal{G}(z,\varepsilon,C,\ell)}\geq 1-\frac{1}{c}2^{-c2^{\sqrt{\log{\norm{z}{}}}}}.
\end{equation}
\end{proposition}

We introduce some tools from \cite{Shadow2}.
\begin{definition}
    Let $\chi : \R^2 \to [0,1]$ be a smooth $\mathcal{C}^\infty$ function such that $\chi(z)=1$ for all $\norm{z}{}\leq \frac{1}{4}$ and $\chi(z)=0$ for all $\norm{z}{}\geq \frac{1}{2}$ and such that $\norm{\nabla \chi(z)}{}\leq 10$ for all $z\in \R^2$. Let $R\geq 1$. We define for $z\in \R^2$, $\chi_R(z) := \chi(z/R)$ and we define the field $f_R$ by
    \begin{equation}
        \forall z\in \R^2, f_R(z):= ((q\chi_R)\ast W)(z),
    \end{equation}
    where we recall that the field $f$ was defined by $f(z) = (q\ast W)(z).$
    We also define the field $\alpha_R$ by
    \begin{equation}
        \forall z \in \R^2, \alpha_R(z) := \sup_{0<r<R}\frac{f_R(z+re_1)-f_R(z)}{r}.
    \end{equation}
\end{definition}
It appears that the field $\alpha_R$ is a good local approximation of $\alpha$.
\begin{proposition}[{{\cite[Corollary 2.9]{Shadow2}}}]
\label{prop:alpha_R_good_approximation}
If $f$ satisfies Assumption \ref{a:a1} for some $\beta>1$, then for any $0<b<\min\left(2,\frac{2\beta-2}{3}\right)$ there exist constants $C,c>0$ such that
\begin{equation}
    \forall \varepsilon\in ]0,1[,\ \Proba{\exists z \in [0,1]^2,\ |\alpha(z)-\alpha_R(z)|\geq \varepsilon}\leq Ce^{-cR^b\varepsilon^2}.
\end{equation}
\end{proposition}

\begin{definition}
\label{def:eventARLC}
Let $\ell\in \mathbb{R}_+^*$ and $R\geq 1$ and $C>0$. We say that the event $\mathcal{A}(R,\ell,C)$ holds, if there exists a continuous rectifiable loop of Euclidean length at most $C$ which separates the inner boundary from the outer boundary of the annulus $[-R,2R]^2\setminus [0,R]^2$ and such that this loop takes values in $\{\alpha_R\leq \ell\}\cap [-R,2R]^2\setminus [0,R]^2$.
\end{definition}
We need the following proposition
\begin{proposition}
\label{prop:seed_event}
Suppose that Assumption \ref{a:a1} is verified.
For any $\ell>\ell_c$ and any $\delta>0$, there exists $R_0\geq 1$ such that for any $R\geq R_0$ there exists $C>0$ (depending a priori on $R$) such that
\begin{equation}
    \Proba{\mathcal{A}(R,\ell,C)}\geq 1-\delta.
\end{equation}
\end{proposition}
\begin{proof}
Let $\ell'\in \R$ be such that $\ell_c<\ell'<\ell$. We denote by $\tilde{\mathcal{A}}(R,\ell',C)$ the same event as $\mathcal{A}(R,\ell',C)$ where we replace the field $\alpha_R$ by the field $\alpha$ in Definition \ref{def:eventARLC}. We also denote by $\tilde{\mathcal{A}}(R,\ell',\infty)$ the event that there exists a loop in $\{\alpha\leq \ell'\}$ separating the inner boundary from the outer boundary of $[-R,2R]^2\setminus [0,R]^2$. Then, since $\ell'>\ell_c$, by a classical gluing construction and using Proposition \ref{prop:high_prob_cross}, we see that
\begin{equation}
    \Proba{\tilde{\mathcal{A}}(R,\ell',\infty)}\xrightarrow[R\to \infty]{}1.
\end{equation}
Hence, we may find $R_0\geq 1$ such that for any $R\geq R_0$ we have
\begin{equation}
    \Proba{\tilde{\mathcal{A}}(R,\ell',\infty)}\geq 1-\frac{\delta}{4}.
\end{equation}
Now we see by Proposition \ref{prop:good_level_set}, that almost surely if there exists such a loop in $\{\alpha\leq \ell'\}$ then one may assume that this loop if of finite length (in fact one may assume that this loop is a piecewise $\mathcal{C}^1$ loop). Indeed this relies on the more general following principle. Given two distincts points $a,b\in \R^2$ that we want to connect by a path in $\{\alpha\leq \ell'\}$, we may try first to connect them by a straight segment. However, it may be the case that this segment is not included in $\{\alpha\leq \ell\}$. We may define $z_0$ as the closest point to $b$ of that segment such that the sub-segment $[a,z_0]$ is included in $\{\alpha\leq \ell'\}$. By continuity we have $\alpha(z_0)=\ell'$. By applying Proposition \ref{prop:good_level_set} we see that we may make a detour and follow a piecewise $\mathcal{C}^1$ boundary (in $\{\alpha=\ell'\}$) until we reach the segment $[a,b]$ again at which point we continue to go straight, until we reach the next point $z_1$ defined similarly. We continue to iterate this procedure until we reach $b$. Compacity of the segment $[a,b]$ together with Proposition \ref{prop:good_level_set} ensures that this procedure ends in a finite number of steps if $a$ and $b$ were connected in $\{\alpha\leq \ell'\}$. The path obtained is a connection between $a$ and $b$ in $\{\alpha\leq \ell'\}$ and is of finite length (it is a piecewise $\mathcal{C}^1$ path).

We now continue the proof of Proposition \ref{prop:seed_event}. By the previous argument we have
\begin{equation}
    \Proba{\exists C\in ]0,\infty[, \tilde{\mathcal{A}}(R,\ell',C)}=\Proba{\tilde{\mathcal{A}}(R,\ell',\infty)}.
\end{equation}
Therefore, we see that
\begin{equation}
    \Proba{\tilde{\mathcal{A}}(R,\ell,C)}\xrightarrow[C\to \infty]{}\Proba{\tilde{\mathcal{A}}(R,\ell,\infty)}.
\end{equation}
By taking $C$ big enough (depending a priori on $R\geq R_0$) we find that
\begin{equation}
    \forall R\geq R_0,\ \exists C>0,\ \Proba{\tilde{\mathcal{A}}(R,\ell',C)}\geq 1-\frac{\delta}{2}
\end{equation}
If we now assume that $\alpha$ and $\alpha_R$ differ by at most $\ell-\ell'$ on $[-R,2R]^2$, then a loop in $\{\alpha\leq \ell'\}$ of length less than $C$ is actually also a loop of $\{\alpha_R\leq \ell\}$ of length less than $C$. To obtain the conclusion it is therefore enough to show that the probability that there exists some point $z\in [-R,2R]^2$ such that $\alpha(z)$ and $\alpha_R(z)$ differ by more than $\ell-\ell'$ goes to $0$ when $R$ goes to infinity. This is a direct consequence of the estimate in Proposition \ref{prop:alpha_R_good_approximation} together with an union bound since
\begin{equation}
    \Proba{\exists z \in [-R,2R]^2,\ |\alpha(z)-\alpha_R(z)|>\ell-\ell'}\leq C'R^2e^{-cR^b(\ell-\ell')^2},
\end{equation}
for some constants $C',c,b>0.$
\end{proof}

We now provide the proof of Proposition \ref{prop:global_struct_high_prob}.
\begin{proof}[Proof of Proposition \ref{prop:global_struct_high_prob}.]
The proof is based on a technical renormalization argument. It is to be noted that this type of renormalization argument has already been used in a variety of contexts (see for instance \cite{Szn12}, \cite{ARS14}, \cite{DCGRS23}, \cite{DCRRV23}, \cite{Shadow2} for an incomplete list). We use the notations introduced in \cite{Shadow2}.

Let $\lambda_0\in \mathbb{N}^*$ to be fixed later.
For $n\geq 0$ we define three sequences $(\sigma_n)_n, (\lambda_n)_n, (\mu_n)_n$ of positive integers as follows :
$$
\begin{cases}
    \sigma_n &= 100\\
    \mu_n &= 10^5\cdot 2^n \\
    \lambda_n &= \lambda_{0}\mu_0^n2^{n(n-1)/2}
\end{cases}
$$
We remark that we have in particular the relation
$$\forall n\in \N,\ \lambda_{n+1}=\lambda_n\mu_n.$$
Let $\varepsilon_0>0$ be fixed.

Denote by $\mathbb{L}_n$ the lattice $\mathbb{L}_n := \lambda_n \mathbb{Z}^2.$ When $u\in \mathbb{L}_n$ we define $\Lambda_n(u) := (u+[0,\lambda_n[)^2\cap \mathbb{L}_{n-1}$. We remark that $\Lambda_n(u)$ contains exactly $\mu_{n-1}^2$ points.

Let $\ell>\ell_c$ and let $\ell'$ and $\ell'$ be such that $\ell_c<\ell'<\ell''<\ell$. According to Proposition \ref{prop:alpha_R_good_approximation} and Proposition \ref{prop:seed_event} we may find a constant $R_0\geq 1$ such that for any $R\geq R_0$ there exists a constant $C(R,\varepsilon_0, \ell')>0$ such that, with probability at least $1-\varepsilon_0$, there exists a rectifiable loop in $\{\alpha_R\leq \ell'\}$ of length less than $C(R,\varepsilon_0,\ell')$ that separates the inner boundary and the outer boundary of $[-R,2R]^2\setminus [0,R]^2.$ In the following we choose $\lambda_0\geq R_0$ and we let $C_0=C_0(\lambda_0,\varepsilon_0,\ell')>0$ be the constant associated to $\lambda_0$. For $u\in \mathbb{L}_0$ we denote by $\mathcal{A}^{(0)}_u$ the event that there exists a rectifiable loop in $\{\alpha_{\lambda_0}\leq \ell'\}$ that separates the inner boundary from the outer boundary of $u+[-\lambda_0,2\lambda_0]^2\setminus [0,\lambda_0]^2$ and which is of Euclidean length less than $C_0$.
By stationarity, it holds that for any $u\in \mathbb{L}_0$ we have
\begin{equation}
    \label{eq:seed_event}
    \Proba{\mathcal{A}^{(0)}_u}\geq 1-\varepsilon_0.
\end{equation}

For $n\in \mathbb{N}^*$ and $u\in \mathbb{L}_n$ we introduce the event $\mathcal{B}^{(n)}_u$ which is defined as the event that on the whole box $u+[-\lambda_n,2\lambda_n]^2$, we have $$|\alpha_{\lambda_n}-\alpha_{\lambda_{n-1}}|\leq \frac{\min(\ell''-\ell',1)}{2^n}.$$
By the estimate from Proposition \ref{prop:alpha_R_good_approximation} an by an union bound we find that for $0<b<\frac{2\beta-2}{3}$ we have constants $c,c'>0$ such that
\begin{equation*}
    \forall n\in \mathbb{N}^*,\ \forall u \in \mathbb{L}_n,\ \Proba{\mathcal{B}^{(n)}_u}\geq 1 - c\lambda_n^2e^{-c'4^{-n}\lambda_{n-1}^b}.
\end{equation*}
By taking $\lambda_0$ big enough we find that
\begin{equation}
    \label{eq:auxiliary_event}
    \forall n\in \mathbb{N}^*,\ \forall u \in \mathbb{L}_n,\ \Proba{\mathcal{B}_u^{(n)}}\geq 1-2^{-\frac{1}{\varepsilon_0}2^{n}}.
\end{equation}
In the following we now fix $\lambda_0$ (big enough depending on $\varepsilon_0$) such that \eqref{eq:seed_event} and \eqref{eq:auxiliary_event} hold. Note that the constant $C_0$ now only depends on $\varepsilon_0$.

Now we recursively define the events $\mathcal{A}^{(n)}_u$ the following way,
$$\forall n\in \mathbb{N}^*,\ \forall u\in \mathbb{L}_n,\ \mathcal{A}^{(n)}_u := \mathcal{B}^{(n)}_u \cap \bigcap_{\substack{u_1,u_2 \in \Lambda_n(u)\\ \norm{u_1-u_2}{\infty}\geq 5\sigma_{n-1} \lambda_{n-1}}}(\mathcal{A}^{(n-1)}_{u_1}\cup \mathcal{A}^{(n-1)}_{u_2}),$$
where the big intersection is taken over all  $(u_1,u_2)\in \Lambda_n(u)$ such that the supremum norm of $u_1-u_2$ is at least $5\sigma_{n-1}\lambda_{n-1}.$

According to Proposition 3.10 in \cite{Shadow2}, we may choose $\varepsilon_0$ small enough (depending only on the sequence $\mu$ which is already fixed and independent from anything else), and we have
\begin{equation}
    \label{eq:An_event}
    \forall n\in \mathbb{N}^*,\ \forall u\in \mathbb{L}_n, \ \Proba{\mathcal{A}^{(n)}_u}\geq 1-2^{-2^n}.
\end{equation}
Recall that now that $\varepsilon_0$ is fixed, then so are the constants $\lambda_0$ and $C_0$.

We now let $x\geq (2\lambda_0)^{\frac{1}{\varepsilon}}$ and $z=(x,0)$ (the more general case $z\in \R^2$ can be dealt with in the same manner). Let $\varepsilon\in ]0,1[$ be fixed. Let $N(x)\in \mathbb{N}$ be the biggest integer such that $\lambda_{N(x)}\leq \frac{1}{2}x^\varepsilon.$ Note that due to explicit expression of $\lambda_n$ we see that we have $N(x)\geq c\sqrt{\ln(x)}$ for some small constant $c>0$ depending on $\varepsilon$. We consider the points $u_k = (k\lambda_{N(x)},0)\in \mathbb{L}_{N(x)}$ for $-1\leq k\leq \frac{x}{\lambda_{N(x)}}+1.$ We note that there is less than $x$ such integers $k$.
Let $\mathcal{G}_x$ be the event that all $\mathcal{A}^{N(x)}_{u_k}$ occur for $-1\leq k \leq \frac{x}{N(x)}+1.$
By \eqref{eq:An_event} and an union bound and adjusting constants, we find that there exists $c>0$ such that
\begin{equation}
    \label{eq:event_gx}
    \Proba{\mathcal{G}_x}\geq 1-\frac{1}{c}2^{-c2^{\sqrt{\log(\norm{z}{})}}}.
\end{equation}
We claim that on the event $\mathcal{G}_x$ the event $\mathcal{G}(z,\varepsilon, C', \ell)$ occurs (with some constant $C'$ that will depend on the constant $C_0>0$ which was fixed in terms of $\varepsilon_0$ and is independent from $x$).
In fact, according to Lemma 5.3 in \cite{ARS14} if all the events $\mathcal{A}^{N(x)}_{u_k}$ occur, then we may find a nearest neighbor path $v_1,\dots,v_m$ in $\mathbb{L}_0$ such that :
\begin{enumerate}
    \item $\norm{v_1-0}{}\leq 2\lambda_{N(x)} \leq x^\varepsilon,$
    \item $\norm{v_m-z}{}\leq 2\lambda_{N(x)}\leq x^\varepsilon$
    \item For each $1\leq i \leq m$, there exists a sequence $(v_i^l)_{1\leq l\leq N(x)}$ such that $v_i^l\in \mathbb{L}_l$ with $v_i\in v_i^l+[0,\lambda_l[^2$ and such that all the events $\mathcal{B}^{(l)}_{v_i^l}$ occurs (for $1\leq i \leq m$ and $1\leq l \leq n$).
    \item For each $1\leq i \leq m$, the event $\mathcal{A}^{(0)}_{v_i}$ occurs.
    \item There exists a universal constant $c>0$ such that $$m\leq (k+2)\prod_{j=0}^{N(x)-1}\mu_j\left(1+\frac{c\sigma_j}{\mu_j}\right).$$
\end{enumerate}
Note by that in the fifth point we have by our definition of $\mu_j$ and $\sigma_j$ that $\prod_{j=0}^\infty \left(1+\frac{c\sigma_j}{\mu_j}\right)<+\infty.$
Moreover we have $k+2\leq \frac{2x}{\lambda_{N(x)}}$ and $\prod_{j=0}^{N(x)-1}\mu_j = \frac{\lambda_{N(x)}}{\lambda_0}.$
Therefore we have found a constant $C_1$ depending only on $\lambda_0$ (which is fixed) such that
\begin{equation}
    m\leq C_1x.
\end{equation}
Now observe that all the events $\mathcal{A}^{(0)}_{v_i}$ occur for $0\leq i \leq m$. In particular, for every such $i$ there exists a loop $\gamma_i$ in $\{\alpha_{\lambda_0}\leq \ell'\}$ of length less than $C_0$ which is included in a box of scale $\lambda_0$ centered around $v_i$. Since $v_i$ and $v_{i+1}$ are two neighbours in $\mathbb{L}_0$ then the two curve $\gamma_i$ and $\gamma_{i+1}$ intersect each other. Therefore, we obtain a macroscopic connection, that is, a continuous path in $\{\alpha_{\lambda_0}\leq \ell'\}$ that starts at distance less than $x^\varepsilon$ of the origin and ends at distance less than $x^\varepsilon$ of $z$ and such that the length of this path is at most $mC_0\leq C_0C_1x$ where $C_0>0$ was the constant introduced in the definition of the events of the form $\mathcal{A}^{(0)}_u$. Moreover this path is included in the union of the boxes of the form $v_i+[-\lambda_0,2\lambda_0]^2$. By the occurrence of the events $\mathcal{B}_{v_i^l}^{(l)}$, we deduce that this path which is included in $\{\alpha_{\lambda_0}\leq \ell'\}$ is also included in $\{\alpha_{\lambda_{N(x)}}\leq \ell''\}$. However by Proposition \ref{prop:alpha_R_good_approximation}, we may bound from above the probability that $\alpha$ and $\alpha_{\lambda_{N(x)}}$ differ by more than $\ell-\ell''$ in the union of the boxes of the form $v_i+[-\lambda_0,2\lambda_0]^2$. More precisely, there is a constant $c>0$ (depending on $\lambda_0$ and $\ell-\ell'$) such that this probability is at most
$$\frac{1}{c'}e^{-c'\lambda_{N(x)}^b}\leq \frac{1}{c}2^{-c2^{\sqrt{\log(x)}}}.$$
Together with \eqref{eq:event_gx} we find a constant $c>0$ such that
\begin{equation}
    \Proba{\mathcal{G}(z,\varepsilon,C_0C_1, \ell)} \geq 1-\frac{1}{c}2^{-c2^{\sqrt{\log(\norm{z}{})}}}, 
\end{equation}
which concludes the proof in the case $z=(x,0)$ and $x\geq \lambda_0$. Adjusting constants we easily show that this also holds for $x\geq 1$. And a completely similar argument shows that this holds for a general $z\in \R^2$ with $\norm{z}{}\geq 1.$
\end{proof}

\subsection{Proof of Theorem \ref{thm:principal}}
We now provide the proof of our main theorem.
\begin{proof}[Proof of Theorem \ref{thm:principal}]
Let $\ell>\ell_c$ such that $\ell\in \mathcal{L}$ where $\mathcal{L}$ is the set of full measure given by Proposition \ref{prop:kac_rice_alpha}. We will find a constant $\tilde{C}>0$ such that if $z\in \R^2$ is connected to $0\in \R^2$ by a path in $\{\alpha\leq \ell\}$ then, with high probability, this connection can actually be made of Euclidean length less than $\tilde{C}\norm{z}{}$. Let $\varepsilon\in ]0,\frac{1}{2}[$ be fixed. Let $C>0$ be the constant given by Proposition \ref{prop:global_struct_high_prob} such that the event $\mathcal{G}(z,\varepsilon,C,\ell)$ has high probability when $\norm{z}{}$ goes to infinity. Let $z\in \R^2$ such that $\norm{z}{}\geq 10$. For $r\geq 0$, Let $B_z^1(r)$ and $B_z^2(r)$ be the two following boxes of $\R^2$.
\begin{align*}
    B_z^1(r) &= [-r\norm{z}{}^\varepsilon,+r\norm{z}{}^\varepsilon]^2 \\
    B_z^2&= z+B_x^1.
\end{align*}
Let us denote by $\mathcal{A}^1_z$ (resp. $\mathcal{A}^2_z$) the event that there exists a loop in the annulus $B_z^1(2)\setminus B_z^1(1)$ (resp. $B_z^2(2)\setminus B_z^2(1)$) included in $\{\alpha\leq \ell\}$ that separates the inner boundary from the outer boundary of this annulus. Since $\ell>\ell_c$, then, by classical gluing constructions and Proposition \ref{prop:high_prob_cross}, we have
\begin{equation}
    \label{eq:proba_Aiz}
    \forall i\in \{1,2\},\ \Proba{\mathcal{A}^i_z}\geq 1-\frac{1}{c}e^{-c\norm{z}{}^\varepsilon}.
\end{equation}
When $i\in \{1,2\}$ we also denote by $\mathcal{B}^1_z$ the event

$$\mathcal{B}^i_z := \left\{S^{\{\alpha\leq \ell\}}(B_z^i(2)\cap \{\alpha\leq \ell\})\leq \norm{z}{}\right\},$$
where we recall Definition \ref{def:chem} of $S^E(D)$.

According to Lemma 3.2 in \cite{Vernotte1} we have the following deterministic control
\begin{equation}
    S^{\{\alpha\leq \ell\}}(B_z^i(2)\cap \{\alpha\leq \ell\}) \leq c\norm{z}{}^\varepsilon+\sigma_\ell(\alpha, B_z^i(2)),
\end{equation}
where $c>0$ is a universal constant and where we recall that $\sigma_\ell(\alpha,B)$ denotes the one dimensional Hausdorff measure of $\{\alpha=\ell\}\cap B$ (which coincides with its Euclidean length according to Proposition \ref{prop:good_level_set}).

By Proposition \ref{prop:kac_rice_alpha}, and by covering the box $B_z^i(2)$ by $O(\norm{z}{}^{2\varepsilon})$ boxes of size $1$ we see that there exists a constant $c'>0$ such that
\begin{equation}
    \forall i \in \{1,2\},\ \mathbb{E}\left[S^{\{\alpha\leq \ell\}}(B_z^i(2)\cap \{\alpha\leq \ell\})\right] \leq c'\norm{z}{}^{2\varepsilon}.
\end{equation}
An application of the Markov inequality then yields
\begin{equation}
    \forall i \in \{1,2\},\ \Proba{\mathcal{B}_z^i}\geq 1 - \frac{c'}{\norm{z}{}^{1-2\varepsilon}}.
\end{equation}
Let us denote by $\mathcal{C}_z$ the event that $0$ and $z$ are connected by a path in $\{\alpha\leq \ell\}.$
We claim that, under the event $\mathcal{C}_z\cap \mathcal{G}_z\cap \mathcal{A}_z^1\cap \mathcal{A}_z^2\cap \mathcal{B}_z^1 \cap \mathcal{B}_z^2$, then there exists a path from $0$ to $z$ in $\{\alpha\leq \ell\}$ of length at most $(C+2)\norm{z}{}$. In fact, we know by the event $\mathcal{G}_z$ that there exists a path $\gamma$ in $\{\alpha\leq \ell\}$ such that one extremity of this path belongs to $B_z^1(1)$ and the other belongs to $B_z^2(1).$ By the events $\mathcal{A}_z^1$ and $\mathcal{A}_z^2$ we may find two loops $\gamma_1$ and $\gamma_2$ such that $\gamma_i$ is a loop in $\{\alpha\leq \ell\}$ that separates the inner boundary from the outer boundary of the annulus $B_z^i(2)\setminus B_z^i(1)$. By planarity the path $\gamma$ must intersect the loop $\gamma_1$ at some $z_1$ and the loop $\gamma_2$ at some $z_2$. Under the event $\mathcal{C}_z$ we have a path $\gamma'$ from $0$ to $z$ in $\{\alpha\leq \ell\}$. Again by planarity this path must intersects $\gamma_1$ at some $z_1'$ and $\gamma_2$ at some $z_2'$. We then obtain a path $\gamma''$ from $0$ to $z$ in the following way, we start from $0$ we follow the path $\gamma'$ until we first intersect $\gamma_1$ at $z'_1$ we then move along the loop $\gamma_1$ to reach $z_1$, we then move along the path $\gamma$ to reach $\gamma_2$ at $z_2$, we travel along the loop $\gamma_2$ until we reach $z'_2$ and then follow the path $\gamma'$ to reach $z$.
Under the events $\mathcal{B}_z^1$ and $\mathcal{B}_z^2$ the Euclidean length needed to connect $0$ to $z_1$ and $z$ to $z_2$ by this path is at most $2\norm{z}{}$. On the event $\mathcal{G}_z$, the Euclidean length needed to connect $z_1$ to $z_2$ by $\gamma''$ is at most $C\norm{z}{}$. Therefore, the total length of $\gamma''$ is at most $(2+C)\norm{z}{}$. Doing an union bound we find that the probability that $0$ is connected to $z$ by a path in $\{\alpha\leq \ell\}$ but such a connection cannot be made with a path of Euclidean length less than $(2+C)\norm{z}{}$ is at most of order $\frac{1}{\norm{z}{}^{1-2\varepsilon}}.$ This concludes the proof of Theorem \ref{thm:principal}.
\end{proof}

\bibliography{biblio.bib}

\noindent
David, Vernotte
\\
Univ. Grenoble Alpes, CNRS, IF, 38000 Grenoble, France

\end{document}